\documentclass[11pt,leqno]{amsart}

\usepackage[colorlinks=false,hidelinks]{hyperref} 
\usepackage[T1,T5]{fontenc} 
\usepackage{amssymb}
\usepackage{amsrefs}
\usepackage[shortlabels]{enumitem}
\usepackage{longtable}

\newtheorem{thm}{Theorem}[section]
\newtheorem{lemma}[thm]{Lemma}
\newtheorem{cor}[thm]{Corollary}
\newtheorem{prop}[thm]{Proposition}
\newtheorem{conj}[thm]{Conjecture}

\theoremstyle{definition}
\newtheorem{definition}[thm]{Definition}
\newtheorem{example}[thm]{Example}
\newtheorem{prob}[thm]{Problem}

\theoremstyle{remark}
\newtheorem{rmk}[thm]{Remark}
\newtheorem{obs}[thm]{Observation}
\newtheorem*{notation}{Notation}

\numberwithin{equation}{section}

\newcommand{\op}{^{\mathrm{op}}}
\DeclareMathOperator{\Good}{Good}
\DeclareMathOperator{\id}{id}
\newcommand{\Z}{\mathbb{Z}}
\newcommand{\F}{\mathcal{F}}
\newcommand{\inv}{^{-1}}
\DeclareMathOperator{\Aut}{Aut}
\DeclareMathOperator{\Inn}{Inn}
\DeclareMathOperator{\Conj}{Conj}
\DeclareMathOperator{\Alex}{Alex}
\DeclareMathOperator{\Free}{Free}
\newcommand{\oo}{\mathcal{O}}
\renewcommand{\phi}{\varphi}
\newcommand{\surj}{\twoheadrightarrow}
\newcommand{\bij}{\xrightarrow{\sim}}
\DeclareMathOperator{\Fix}{Fix}

\makeatletter
\@namedef{subjclassname@2020}{\textup{2020} Mathematics Subject Classification}
\makeatother

\makeatletter
\newcommand{\addresseshere}{%
	\enddoc@text\let\enddoc@text\relax
}
\makeatother

\hypersetup{
	pdflang={en-US},
	pdftitle={Good involutions of twisted conjugation subquandles and Alexander quandles},
	pdfauthor={Lực Ta},
	pdfsubject={Mathematics},
	pdfkeywords={Alexander quandle, enumeration, good involution, group automorphism, linear quandle, symmetric quandle, twisted conjugation quandle}
}

\begin{document}
	
	\title[
    Good involutions of twisted conjugation subquandles
    ]{
    Good involutions of twisted conjugation subquandles\\and Alexander quandles
    }
	\author{L\d\uhorn c Ta}	
	
	\address{Department of Mathematics, University of Pittsburgh, Pittsburgh, Pennsylvania 15260}
	\email{ldt37@pitt.edu}
	
	
	
	\begin{abstract}
	We completely describe good involutions of free quandles and subquandles of twisted conjugation quandles of groups, including all Alexander quandles. As an application, we enumerate good involutions of linear quandles, and we provide explicit mappings for those up to order $23$ via a computer search. Along the way, we completely characterize connected, involutory Alexander quandles, which may be of independent interest.
	\end{abstract}
	\maketitle
	
\section{Introduction}\label{sec:intro}

Introduced by Kamada \cite{symm-quandles-2} in 2007, \emph{symmetric quandles} are nonassociative algebraic structures used to study surface-links \citelist{\cite{symm-quandles-2}\cite{tensor}\cite{yasuda}\cite{dihedral}}, handlebody-links \cite{multiple}, $3$-manifolds \cite{nosaka}, compact surfaces with
boundary in ribbon forms \cite{symm-racks}, and framed virtual knots in thickened surfaces \cite{virtual}. Defined as pairs $(Q,\rho)$ where $Q$ is a quandle and $\rho$ is a function called a \emph{good involution} of $Q$, symmetric quandles enjoy various (co)homology theories \citelist{\cite{symm-quandles-2}\cite{karmakar}} and connections to group theory \citelist{\cite{lie}\cite{multiple}}. 

Symmetric quandles can be traced back to \emph{kei}, which Takasaki~\cite{takasaki} introduced in 1943 to study Riemannian symmetric spaces; \emph{quandles}, which Joyce~\cite{joyce} and Matveev~\cite{matveev} independently introduced in 1982 to construct complete invariants of unframed links; and \emph{racks}, which Fenn and Rourke \cite{fenn} introduced in 1992 to construct complete invariants of framed links. See \citelist{\cite{quandlebook}\cite{book}} for references on the theory and \cite{survey} for a survey of the algebraic state of the art.

Building upon the results in \cite{ta-conj}, this paper concerns the following classification problems posed by Taniguchi \cite{taniguchi}.

\begin{prob}\label{prob1}
    Find necessary and sufficient conditions for good involutions $\rho$ of quandles to exist.
\end{prob}

\begin{prob}\label{prob2}
    Given a quandle $Q$, determine all symmetric quandles of the form $(Q,\rho)$.
\end{prob}

Although various constructions of symmetric quandles exist \citelist{\cite{galkin-symm}\cite{dihedral}\cite{symm-quandles}}, few complete classification results in the vein of Problems \ref{prob1} and \ref{prob2} exist in the literature. 
In 2010, Kamada and Oshiro \cite{symm-quandles-2} addressed Problem \ref{prob1} for kei and Problem \ref{prob2} for dihedral quandles and trivial quandles. 
In 2023, Taniguchi \cite{taniguchi} solved Problem \ref{prob1} for generalized Alexander quandles and Problem \ref{prob2} for connected generalized Alexander quandles. In 2025, the author \cite{ta-conj} addressed Problems \ref{prob1} and \ref{prob2} for core quandles and subquandles of conjugation quandles.

This paper answers Problems \ref{prob1} and \ref{prob2} for a large class of quandles called \emph{twisted conjugation quandles}, which Andruskiewitsch and Gra\~na \cite{hopf} introduced in 2003 under the name \emph{twisted homogeneous crossed sets} to study pointed Hopf algebras. Twisted conjugation quandles are of considerable interest to algebraists and knot theorists alike \citelist{\cite{embed}\cite{residual}\cite{sc}\cite{hopf}}, in no small part because they generalize two already large classes of quandles called \emph{conjugation quandles} and \emph{Alexander quandles}. 

Conjugation quandles are used to generalize Wirtinger presentations of knot groups \citelist{\cite{joyce}\cite{matveev}}, and they enjoy many other connections across algebra \citelist{\cite{lattice}\cite{conjugation}\cite{lie}\cite{taGQ}} and low-dimensional topology \citelist{\cite{eisermann}\cite{fenn}}. In particular, good involutions of conjugation quandles are used to study noncommutative geometry applied to finite groups \cite{lie}, minimal triple point numbers of nonorientable surface-knots \cite{dihedral}, and coloring invariants of handlebody-knots \cite{multiple}. Many common classes of quandles canonically embed into conjugation quandles; see \cite{embed} for a list. On the other hand, Alexander quandles provide a framework for studying modules over the ring of integral Laurent polynomials $\Z[t^{\pm 1}]$ \citelist{\cite{alex}\cite{hou}}, and they provide linear-algebraic generalizations of Fox coloring invariants of knots \cite{fox}. 

\subsection{Main result}
Our most general result states that, given a subquandle $X$ of a twisted conjugation quandle $\Conj(G,\phi)$, every good involution $\rho$ of $X$ is completely determined by a unique function $\zeta^*$ from the set $\oo(X,s)$ of connected components of $(X,s)$ to a certain set $S$ (cf.\ Remark \ref{rmk:core}) satisfying certain conditions. 

\begin{thm}[Theorem \ref{thm:main2}]\label{thm:main}
    Let $\phi$ be an automorphism of a group $G$, and let $(X,s)$ be a subquandle of the twisted conjugation quandle $\Conj(G,\phi)$. Let $\langle X\rangle$ denote the subgroup of $G$ generated by $X$, and define the set
    \[
    S:=\{x\in \langle X\rangle \mid \phi^{2}(y)=\phi(x)yx\inv \text{ for all }y\in  X\}.
    \]
    (See \eqref{eq:s2} for another description of $S$ in the case that $X$ contains the identity element of $G$.) 
    
    Then for all functions $\rho:X\to X$, $\rho$ is a good involution of $(X,s)$ if and only if there exists a function $\psi:X\to S$ such that
    \[
    \rho(x)=\psi(x)\phi(x\inv),\qquad\psi s_x=\psi=\psi\rho
    \]
    for all $x\in X$. (Cf.\ Remark \ref{rmk:core} and Theorem \ref{cor:taniguchi}.)
\end{thm}

Theorem \ref{thm:main} strengthens several similar results in the literature on symmetric quandles.

\begin{example}
    Let $\phi:=\id_G$ be the identity map. Then $S=Z(\langle X\rangle)$, and $\Conj(G,\id_G)=(\Conj G)\op$ is isomorphic to $\Conj G$ via the inversion map $g\mapsto g\inv$. Therefore, Theorem \ref{thm:main} specializes to the statement of \cite{ta-conj}*{Thm.\ 1.3} with $X$ replaced by $X\inv$. Since we use the latter theorem to prove Theorem \ref{thm:main}, the two results are actually equivalent.
\end{example}

\begin{example}   
    If $G$ is abelian, then $\Conj(G,\phi)$ equals the Alexander quandle $\Alex(G,\phi)$. 
    Thus, Theorem \ref{thm:main} allows us to strengthen certain classification results of Taniguchi \cite{taniguchi} in the abelian case; see Section \ref{sec:alex}. 
    
    In particular, unlike some of the results in \cite{taniguchi}, Theorem \ref{cor:taniguchi} is not limited to connected Alexander quandles. We take advantage of this freedom to study good involutions of \emph{linear quandles}; see Section \ref{sec:linear}. In particular, we recover a result of Kamada and Oshiro \cite{symm-quandles-2}*{Thm.\ 3.2} for \emph{dihedral quandles}.
\end{example}

\subsection{Structure of the paper}
In Section \ref{sec:prelims}, we define good involutions of quandles. We provide several examples and preliminary results. In particular, we show that nontrivial free quandles have no good involutions (Example~\ref{ex:free-none}). This addresses Problems \ref{prob1} and \ref{prob2} for free quandles.

In Section \ref{sec:core}, we prove Theorem \ref{thm:main} (Theorem \ref{thm:main2}) using a result for conjugation quandles from \cite{ta-conj}. This addresses Problems \ref{prob1} and \ref{prob2} for twisted conjugation quandles. As an application, we obtain upper and lower bounds on the number of good involutions of certain subquandles of twisted conjugation quandles (Corollaries \ref{cor:good-bound}--\ref{cor:bnd1} and \ref{cor:conn-const}).

In Section \ref{sec:alex}, we completely describe good involutions of Alexander quandles (Theorem \ref{cor:taniguchi}) by reducing to a special case of Theorem \ref{thm:main}. This addresses Problems \ref{prob1} and \ref{prob2} for Alexander quandles, extending the results of \cite{taniguchi} in the abelian case without the assumption of connectedness. However, by adding this assumption, we strengthen another result of Taniguchi in the abelian case (Corollary \ref{cor:good-taka}). Along the way, we completely describe connected, involutory Alexander quandles (Theorem \ref{thm:conn-q}), which may be of independent interest.

In Section \ref{sec:linear}, we discuss Problems \ref{prob1} and \ref{prob2} for \emph{linear quandles}, a certain class of Alexander quandles having linear-algebraic applications to knot theory \cite{fox}. In particular, we compute the number of good involutions of the linear quandle $\Z[t^{\pm 1}]/(4n,t-2n-1)$, which grows arbitrarily large as $n$ tends toward infinity (Theorem \ref{thm:diverge}). This divergence is surprising given the usual difficulty of finding nonidentity good involutions of quandles (see, for example, \cite{galkin-symm}*{Sec.\ 5}). We also use a computer search to enumerate good involutions of nontrivial linear quandles up to order $29$ (Table \ref{tab:linear}) and provide explicit mappings for those up to order $23$ in a GitHub repository \cite{code}.

\begin{notation}
    Let $X$ be a set, and let $G$ be a group. This paper uses the following notation. \begin{itemize}
        \item Denote the composition of functions $\phi:X\to Y$ and $\psi:Y\to Z$ by $\psi\phi$. 
        \item Let $\id_X:X\to X$ be the identity map of $X$, and let $S_X$ be the symmetric group of $X$.
        \item If $X$ is contained in $G$, let $\langle X\rangle$ be the subgroup of $G$ generated by $X$, and define the set $X\inv:=\{x\inv\mid x\in X\}$.
        \item Denote the center of $G$ by $Z(G)$.
        \item Given a group automorphism $\phi\in\Aut G$, let $\Fix \phi\leq G$ be the subgroup of fixed points of~$\phi$.
    \end{itemize}
\end{notation}

\subsection*{Acknowledgments}
I thank Samantha Pezzimenti, Jose Ceniceros, and Peyton Wood for respectively introducing me to knot theory, quandles, and racks.

\section{Preliminaries}\label{sec:prelims}

\subsection{Quandles} We define quandles and give several examples.

\begin{definition}
    Let $X$ be a set, let $s:X\to S_X$ be a function, and write $s_x:=s(x)$ for all elements $x\in X$. We call the pair $(X,s)$ a \emph{quandle} if \begin{equation*}
		s_xs_y=s_{s_x(y)}s_x,\qquad s_x(x)=x
	\end{equation*}
    for all $x,y\in X$. We call $s$ the \emph{quandle structure} of $(X,s)$, and we say that $|X|$ is the \emph{order} of $(X,s)$. Some authors call the permutations $s_x$ \emph{point symmetries} of $(X,s)$.
\end{definition}

\begin{rmk}
    Some authors define quandles as sets $X$ with equipped with a nonassociative binary operation $\triangleright:X\times X\to X$ satisfying right-distributivity and other axioms. These two definitions are equivalent \citelist{\cite{survey}\cite{joyce}} via the formula
    \[
    s_y(x)=x\triangleright y.
    \]
    Some other authors define quandles using a left-distributive binary operation $\triangleright:X\times X\to X$. Of course, the left-distributive and right-distributive definitions are equivalent.
\end{rmk}

\begin{definition}
    Let $Q:=(X,s)$ be a quandle.
    \begin{itemize}
        \item If $Q$ is a quandle and $Y\subseteq X$ is a subset such that $s_y(Y)= Y$ for all $y\in Y$, then we say that the quandle $(Y,s|_Y)$ is a \emph{subquandle} of $Q$.\footnote{Some authors merely require that subquandles $Y$ satisfy $s_y(Y)\subseteq Y$ for all $y\in Y$. However, if $Y$ is an infinite set, then this weaker definition does not guarantee that $(Y,s|_Y)$ is a quandle; see \cite{cong}*{Subsec.\ 2.3}.} When there is no room for confusion, we may simply say that $Y$ is a subquandle of $Q$.
        \item We say that $Q$ is \emph{involutory} or a \emph{kei} if $s_x^2=\id_X$ for all $x\in X$. 
    \end{itemize}
\end{definition}

\begin{example}[\cite{fenn}]
    Given a set $X$, let $s:X\to S_X$ be the constant function $x\mapsto \id_X$. Then $(X,s)$ is a quandle called a \emph{trivial quandle}.
\end{example}

    \begin{example}[\citelist{\cite{joyce}\cite{fenn}\cite{conjugation}}]\label{def:conj}
    Let $G$ be a group, and define a quandle structure $s:G\rightarrow S_G$ on $G$ by \[s_g(h):= ghg\inv\]
    for all $g,h\in G$. 
	Then $\Conj G:=(G,s)$ is a quandle called a \emph{conjugation quandle} or \emph{conjugacy quandle}. 
    
    Note that $s_g\inv=s_{g\inv}$ for all $g\in G$, and $\Conj G$ is trivial if and only if $G$ is abelian. Moreover, a subset $X\subseteq G$ is a subquandle of $\Conj G$ if and only if $X$ is closed under conjugation by elements of the subgroup $\langle X\rangle\leq G$; see \cite{conjugation}*{Thm.\ 3.1} and cf.\ Example \ref{ex:free}.
\end{example}

\begin{rmk}
    Some authors define conjugation quandles via the formula $s_g(h):=g\inv hg$. This yields the \emph{dual quandle} $(\Conj G)\op$ of the conjugation quandle $\Conj G$ in Definition \ref{def:conj}; see Example \ref{ex:dual-twist-1}.
\end{rmk}

    \begin{example}[\citelist{\cite{embed}\cite{hopf}}]
    Let $G$ be a group, and let $\phi\in\Aut G$ be an automorphism of $G$. Define a quandle structure $s:G\to S_G$ on $G$ by \[s_g(h):= \phi(g\inv h)g\]
    for all $g,h\in G$. 
    Then $(G,s)$ is a quandle called a \emph{twisted conjugation quandle} or \emph{twisted homogeneous crossed set} and denoted by $\Conj(G,\phi)$. 
\end{example}

\begin{rmk}
    For the precise relationship between conjugation quandles $\Conj G$ and twisted conjugation quandles $\Conj(G,\phi)$, see Example \ref{ex:dual-twist}.
\end{rmk}

\begin{example}[\cite{alex}]
    If $A$ is an additive abelian group, then twisted conjugation quandles $\Conj(A,\phi)$ are called \emph{Alexander quandles} or \emph{affine quandles} and denoted by $\Alex(A,\phi)$. The theory and applications of Alexander quandles are well-developed; see, for example,  \citelist{\cite{alex}\cite{hulpke}\cite{cong}\cite{fox}\cite{sim}\cite{hou}}.

    In particular, let $\phi$ be the inversion map $a\mapsto -a$. Then $\Alex(A,\phi)$ is a kei called a \emph{Takasaki kei} and denoted by $T(A)$. Moreover, if $A$ is cyclic of order $n$, then we call $T(\Z/n\Z)$ a \emph{dihedral quandle} and denote it by $R_n$.
\end{example}

\begin{rmk}[\citelist{\cite{cong}\cite{alex}\cite{hou}}]\label{rmk:module}
    Consider the ring $\Z[t^{\pm 1}]$ of integral Laurent polynomials. 
    Given an additive abelian group $A$, a choice of automorphism $\phi\in\Aut A$ is equivalent to a choice of a $\Z[t^{\pm 1}]$-module structure on $A$, namely $ta:=\phi(a)$ for all $a\in A$. 
    
    Authors often alternate between the group-theoretic and module-theoretic points of view when studying Alexander quandles. Indeed, isomorphisms of $\Z[t^{\pm 1}]$-modules induce isomorphisms of Alexander quandles, but not vice versa; see \cite{alex}.
\end{rmk}

\begin{example}[\citelist{\cite{taniguchi}\cite{hopf}}]\label{ex:alex}
    Let $G$ be a group, and let $\phi\in\Aut G$ be an automorphism of $G$. The \emph{generalized Alexander quandle} $\operatorname{GAlex}(G,\phi)$, which Andruskiewitsch and Gra\~na \cite{hopf} introduced under the name \emph{principal homogeneous crossed sets}, is defined to be the quandle with underlying set $G$ and quandle structure
    \[
    s_g(h):=\phi(hg\inv)g.
    \]
\end{example}

\begin{rmk}[\cite{embed}]
    Like with twisted conjugation quandles $\Conj(G,\phi)$, taking $G$ to be abelian recovers the definition of an Alexander quandle. However, $\operatorname{GAlex}(G,\phi)$ and $\Conj(G,\phi)$ are generally distinct.
\end{rmk}

\begin{example}[\citelist{\cite{joyce}\cite{fenn}\cite{free}}]\label{ex:free}
    Let $X$ be a set. The \emph{free quandle} generated by $X$, denoted by $\Free X$, is defined to be the image of $X$ under the left adjoint of the forgetful functor from the category of quandles to the category of sets. For example, if $|X|\leq 1$, then $|{\Free X}|=|X|$.

    In 1982, Joyce \cite{joyce}*{Thm.\ 4.1} gave the following explicit construction of $\Free X$. Let $F(X)$ be the free group generated by $X$, and let $\widetilde{X}$ be the set of conjugates of elements of $X$ in $F(X)$. Then $\widetilde{X}$ is a subquandle of $\Conj F(X)$, and $\widetilde{X}\cong \Free X$. In particular, $\Free X$ is a trivial quandle if and only if $|X|\leq 1$.
\end{example}

\subsection{Quandle homomorphisms}
Quandles form a category with the following morphisms. 

\begin{definition}
    Given two quandles $(X,s)$ and $(Y,t)$, we say that a map $\varphi:X\to Y$ is a \emph{quandle homomorphism} if \[\varphi s_x=t_{\varphi(x)}\varphi\] for all $x\in X$. A \emph{quandle isomorphism} is a bijective quandle homomorphism, and a \emph{quandle automorphism} of is an isomorphism from a quandle to itself. Let $\Aut (X,s)$ denote the group of quandle automorphisms of $(X,s)$.
\end{definition}

\begin{definition}
    Given a quandle $Q=(X,s)$, we define the \emph{inner automorphism group}\footnote{Depending on choices of convention, some authors call $\Inn Q$ the \emph{right multiplication group}, \emph{left multiplication group}, or \emph{operator group} of $Q$.} of $Q$, denoted by $\Inn Q$, to be the subgroup of $S_X$ generated by the set $s(X)$:
    \[
    \Inn Q:=\langle s_x\mid x\in X\rangle.
    \] 
    In fact, it is straightforward to verify that $\Inn Q$ is a normal subgroup of $\Aut Q$.
\end{definition}

\begin{example}[\cite{fenn}*{Sec.\ 1.3}]
    Let $G$ be a group, let $\Conj X$ be a subquandle of $\Conj G$, and let $H:=\langle X\rangle$ be the subgroup of $G$ generated by $X$. Then $\Inn(\Conj X)$ equals the inner automorphism group $\Inn H$ of the group $H$.
\end{example}

\subsubsection{Connected quandles}

\begin{definition}[\cite{orbits}]
    Given a quandle $Q=(X,s)$, let $\oo(Q)$ denote the set of orbits of $X$ under the action of $\Inn Q$. An element of $\oo(Q)$ is called a \emph{connected component} of $Q$.
\end{definition}

\begin{definition}
    We say that a quandle $Q=(X,s)$ is \emph{connected} or \emph{indecomposable} if $Q$ has only one connected component, that is, if the action of $\Inn Q$ on $X$ is transitive.
\end{definition}

\begin{example}[\cite{conjugation}*{Thm.\ 3.2}]\label{lem:conn}
    If $X$ is a subquandle of a conjugation quandle $\Conj G$, then $X$ is connected if and only if $X$ is a conjugacy class in the subgroup $\langle X\rangle\leq G$. In particular, $\Conj G$ is connected if and only if $G$ is the trivial group.
\end{example}

The following example of connected quandles is a well-known result proven in \cite{hopf}*{Subsec.\ 1.3.8} and reproven in \cite{hulpke}*{Cor.\ 7.2} and \cite{cong}*{Sec.\ 7}.

\begin{example}[\cite{hopf}*{Subsec.\ 1.3.8}]\label{lem:conn-alex}
    Given an Alexander quandle $\Alex(A,\phi)$, define the function $\alpha:A\to A$ by
    \[
    \alpha(a):=a-\phi(a)
    \]
    for all $a\in A$.
    Then $\Alex(A,\phi)$ is connected if and only if $\alpha$ is surjective. For more classification results for connected Alexander quandles, see \citelist{\cite{hou}\cite{alex}\cite{orbits}\cite{hopf}\cite{sim}}.
\end{example}

\subsection{Symmetric quandles} 

This paper studies \emph{symmetric quandles}, which Kamada \cite{symm-quandles} introduced in 2007.

\begin{definition}[\citelist{\cite{symm-quandles}\cite{symm-quandles-2}}]\label{def:symmetric}
    Let $Q=(X,s)$ be a quandle. A \emph{good involution} of $Q$ is an involution $\rho\in S_X$ that satisfies the equalities
    \[
    \rho s_x=s_x\rho,\qquad s_{\rho(x)}=s\inv_x
    \]
    for all $x\in X$. (Note that $\rho$ is not required to be a quandle endomorphism.) We call the pair $(Q,\rho)$ a \emph{symmetric quandle}, and we denote the set of all good involutions of $Q$ by $\Good Q$. 
\end{definition}

\begin{obs}\label{obs:iso-good}
    The set of good involutions is a quandle invariant.
    That is, if $\theta:Q_1\bij Q_2$ is a quandle isomorphism, then the assignment $\rho\mapsto \theta\rho\theta\inv$ is a bijection from $\Good Q_1$ to $\Good Q_2$.
\end{obs}

\begin{example}[\cite{symm-quandles-2}*{Prop.\ 3.1}]\label{ex:trivial}
    Given a quandle $Q=(X,s)$, let $I\subset S_X$ the set of all involutions of $X$. Then $\Good Q=I$ if and only if $Q$ is a trivial quandle.
\end{example}

\begin{example}[\cite{symm-quandles}*{Ex.\ 2.4}]\label{ex:conjinv}
    Let $G$ be a group. Then the inversion map $g\mapsto g\inv$ is a good involution of $\Conj G$.
\end{example}

\begin{example}
    Let $Q=\mathrm{GAlex}(G,\phi)$ be a generalized Alexander quandle. A theorem of Taniguchi \cite{taniguchi}*{Thm.\ 1.3} states that $\Good Q$ is nonempty if and only if $Q$ is a kei.
\end{example}

To prove Theorem \ref{thm:main}, we use the following characterization of good involutions of subquandles of conjugation quandles from \cite{ta-conj}.

\begin{thm}[\cite{ta-conj}*{Thm.\ 1.3}]\label{thm2}
    Let $H$ be a group, and let $(Y,t)$ be a subquandle of $\Conj H$. 
    Then for all functions $\rho:Y\to Y$, $\rho$ is a good involution of $\Conj Y$ if and only if there exists a function $\zeta:Y\to Z(\langle Y\rangle)$ such that
    \[
    \rho(y)=(\zeta(y)y)\inv,\qquad\zeta t_y=\zeta=\zeta\rho
    \]
    for all $y\in Y$.
\end{thm}

\begin{example}\label{ex:free-none}
    In the setting of Theorem \ref{thm2}, one can show that if $\langle Y\rangle$ is centerless, then
    \[
    \Good (Y,t)=\begin{cases}
        \{\iota|_Y\}&\text{if }Y=Y\inv,\\
        \emptyset &\text{otherwise,}
    \end{cases}
    \]
    where $\iota:H\to H$ denotes the inversion map $h\mapsto h\inv$; see \cite{ta-conj}*{Cor.\ 6.5}.
    
    In particular, this shows that nontrivial free quandles have no good involutions. Let $Y$ be a set, and let $F(Y)$ be the free group generated by $Y$. By Example \ref{ex:free}, we may identify the free quandle $\Free Y$ with the subquandle of $\Conj F(Y)$ whose elements are conjugates of elements of $Y$. In particular, $\langle \Free Y\rangle=F(Y)$. 
    
    Recall that $\Free Y$ is nontrivial if and only if $|Y|\geq 2$. In this case, $\langle \Free Y\rangle$ is centerless. But clearly $Y\neq Y\inv$ as subsets of $F(Y)$, so $\Good (\Free Y)=\emptyset$. In fact, since all subquandles of free quandles are free \cite{free}, this fact is enough to characterize good involutions of all subquandles of free quandles.
\end{example}

\subsubsection{Dual quandles}

\begin{definition}[\cite{cong}]
    Given a quandle $Q=(X,s)$, define $s\inv:X\to S_X$ by $x\mapsto s_x\inv$. Then the quandle $Q\op:=(X,s\inv)$ is a quandle called the \emph{dual quandle} of $R$.
\end{definition}

\begin{example}\label{ex:dual-twist-1}
    The dual quandle $(\Conj G)\op$ of a conjugation quandle $\Conj G$ is the quandle with underlying set $G$ and quandle structure $g\mapsto s_g\inv$, where
    \[
    s_g\inv(h)=s_{g\inv}(h)=g\inv hg
    \]
    for all $g,h\in G$.
\end{example}

\begin{example}\label{ex:dual-twist}
    The dual quandle $(\Conj(G,\phi))\op$ of a twisted conjugation quandle $\Conj(G,\phi)$ has the quandle structure defined by
    \[
    s_g\inv(h)=g\phi\inv(hg\inv)
    \]
    for all $g,h\in G$. In particular, by taking $\phi$ to be the identity map $\id_G$, we obtain
    \[
    (\Conj(G,\id_G))\op=\Conj G,
    \]
    which is also evident from Example \ref{ex:dual-twist-1}.
\end{example}

\begin{example}[\cite{ta-conj}*{Sec.\ 4}]\label{ex:gis-are-antis}
    Given a quandle $Q$, it is straightforward to verify that good involutions $\rho$ of $Q$ are a special class of quandle \emph{antiautomorphisms}, that is, quandle isomorphisms $\rho:Q\bij Q\op$. In particular, not all quandles have good involutions.
\end{example}

\section{Symmetric twisted conjugation subquandles}\label{sec:core}
In this section, we prove Theorem \ref{thm:main} (Theorem \ref{thm:main2}) and several bounds on the number of good involutions of subquandles of twisted conjugation quandles.

Throughout this section, let $G$ be a group, let $\phi\in\Aut G$ be an automorphism of $G$, and let $(X,s)$ be a subquandle of the twisted conjugation quandle $\Conj(G,\phi)$.

\subsection{Setup for the proof}\label{subsec:A} 
To prove Theorem \ref{thm:main}, we apply Observation \ref{obs:iso-good} and Theorem \ref{thm:main2} to a certain subquandle $(Y,t)$ of a conjugation quandle $\Conj K$ such that $(X,s)\cong (Y,t)$.

\subsubsection{Twisted conjugation quandles as conjugation subquandles}

Let $K$ be the semidirect product
\[K:=G\rtimes_\phi \Z
\]
given by the right action $g\cdot n:=\phi^n(g)$ of $\Z$ on $G$.\footnote{Some geometric group theorists call $K$ a \emph{mapping torus}, especially if $G$ is torsion-free; see \cite{mapping}.} Then multiplication and inversion in $K$ are given by
\[
(g,m)(h,n)=(\phi^n(g)h,m+n),\qquad (g,m)\inv=(\phi^{-m}(g\inv),-m),
\]
respectively. 
In 2023, Akita \cite{embed} showed that the function $\eta:G\to K$ defined by
\begin{equation}\label{eq:phi}
    \eta(g):=(g,1)
\end{equation}
is a quandle embedding of $\Conj(G,\phi)$ into $(\Conj K)\op$. Let $\iota:K\to K$ denote the inversion map $(g,m)\mapsto (g,m)\inv$. By Examples \ref{ex:conjinv} and \ref{ex:gis-are-antis}, the composition $\theta:=\iota\eta$ is a quandle embedding of $\Conj(G,\phi)$ into $\Conj K$. 

Henceforth, let $Y:=\theta(X)$ denote the image of $X$ under the quandle embedding $\theta$:
\[
Y=\{(x,1)\inv\mid x\in X\}=\{(\phi\inv(x\inv),-1)\mid x\in X\}.
\] Let $t$ denote the quandle structure of $Y$ when viewed as a subquandle of $\Conj K$. Also, let $H:=\langle Y\rangle$ be the subgroup of $K$ generated by $Y$. For example, if $X=G$, then $H=K$.

In the following, we apply Theorem \ref{thm2} to $(Y,t)$. We then apply Observation \ref{obs:iso-good} to the quandle isomorphism $\theta:X\bij Y$ to prove Theorem \ref{thm:main}. 

\subsubsection{Calculation of the center} To apply Theorem \ref{thm2} to $(Y,t)$, we first compute $Z(H)$.

\begin{prop}\label{prop:zh}
    The center of $H$ is
    \[
    Z(H)=\{(x,n)\in H\mid \phi^n(y)=\phi(x)yx\inv\text{ for all }y\in X\}. 
    \]
    In particular, if $X$ contains the identity element of $G$, then the center of $K$ is
    \[
    Z(H)=\{(x,n)\in (\Fix\phi\cap\langle X\rangle)\times\Z\mid \phi^n(y)=xyx\inv\text{ for all }y\in X\}.
    \]
\end{prop}

\begin{proof}
    Since $Y\inv=\{(x,1)\mid x\in X\}$ generates $H$, verifying that elements of the right-hand side lie in $Z(H)$ is straightforward. Conversely, let $(x,n)\in Z(H)$, so that
    \[
    (\phi(x)y,n+1)=(x,n)(y,1)=(y,1)(x,n)=(\phi^n(y)x,n+1)
    \]
    for all generators $(y,1)\in Y\inv$. Then $\phi^n(y)=\phi(x)yx\inv$ for all $y\in X$, as desired. In particular, if $1\in X$, then $H$ contains the set $\langle X\rangle \times \Z$, so we can take $y:=1$ above to obtain $x\in\Fix\phi$.
\end{proof}

\subsubsection{Preliminary versions of Theorem \ref{thm:main}}

First, we record a special case of Theorem \ref{thm2}.

\begin{lemma}\label{lem:prelim-conj}
    For all functions $\rho':Y\to Y$, $\rho'$ is a good involution of $\Conj Y$ if and only if there exists a function $\zeta:Y\to Z(H)$ such that
    \[
    \rho'((x,1)\inv)=\zeta((x,1)\inv)\inv(x,1),\qquad\zeta t_{(x,1)\inv}=\zeta=\zeta\rho'
    \]
    for all $(x,1)\inv\in Y$.
\end{lemma}

We refine Lemma \ref{lem:prelim-conj} as follows. As in Theorem \ref{thm:main}, define the set
\begin{equation*}
	S:=\{x\in \langle X\rangle \mid \phi^{2}(y)=\phi(x)yx\inv \text{ for all }y\in X\}.
\end{equation*}
In particular, if $X$ contains the identity element $1\in G$, then taking $y:=1$ as in the end of proof of Proposition \ref{prop:zh} shows that
\begin{equation}\label{eq:s2}
    S=\{x\in \Fix \phi\cap\langle X\rangle\mid \phi^2(y)=xyx\inv\text{ for all }y\in X\}.
\end{equation}
Also, define the set
\[
    S':=\{(x,2)\mid x\in S \}\cap H.
\]

\begin{prop}\label{prop:prelim}
    Lemma \ref{lem:prelim-conj} holds when $Z(H)$ is replaced with $S'$.
\end{prop}

\begin{proof}
    By Proposition \ref{prop:zh}, $S'$ is a subset of $Z(H)$, so any function $\zeta:Y\to S'$ satisfying the conditions of Lemma \ref{lem:prelim-conj} induces a good involution $\rho'$ of $(Y,t)$.

    Conversely, let $\rho'$ be a good involution of $(Y,t)$, so that $\rho'$ is induced by a function $\zeta:Y\to Z(H)$ satisfying the conditions of Lemma \ref{lem:prelim-conj}. We have to show that $\zeta(Y)\subseteq S'$. For all elements $(x,1)\inv\in Y$, write
    \[
    (\zeta_x,n_x):=\zeta((x,1)\inv)\in Z(H),
    \]
    By Proposition \ref{prop:zh} and the inclusion $\zeta(Y)\subseteq Z(H)$, showing that $n_x=2$ will also show that $\zeta_x\in S$. Indeed, $Y$ must contain the element
    \begin{align*}
        \rho'((x,1)\inv)&=\zeta((x,1)\inv)\inv(x,1)\\
        &=(\zeta_x,n_x)\inv(x,1)\\
        &=(\phi^{-n_x}(\zeta_x\inv),-n_x)(x,1) \\
        &=(\phi^{-n_x+1}(\zeta_x\inv)x,-n_x+1),
    \end{align*}
    so $n_x=2$. Hence, $\zeta((x,1)\inv)\in S'$; since $(x,1)\inv\in Y$ was arbitrary, the proof is complete.
\end{proof}

\subsection{Main results}
We now prove the main theorem. As before, let $s$ denote the quandle structure of $(X,s)$ as a subquandle of $\Conj(G,\phi)$.

\begin{thm}[Theorem \ref{thm:main}]\label{thm:main2}
    For all functions $\rho:X\to X$, $\rho$ is a good involution of $(X,s)$ if and only if there exists a function $\psi:X\to S$ such that
    \[
    \rho(x)=\psi(x)\phi(x\inv),\qquad\psi s_x=\psi=\psi\rho
    \]
    for all $x\in X$.
\end{thm}

\begin{proof}
    We will show that the claim is equivalent to Proposition \ref{prop:prelim}. 
    Given any function $\psi:X\to S$, define a function $\F(\psi):Y\to S'$ by \[\F(\psi)((x,1)\inv):=(\psi(x),2).\]
    Conversely, given any function $\zeta :Y\to S'$, define a function $\F\inv(\zeta ):X\to S$ by sending $g\in G$ to the first coordinate of $\zeta ((x,1)\inv)\in S'$. 
    Clearly, the assignments $\F$ and $\F\inv$ are mutually inverse.

    Now, recall that the function $\theta=\iota\eta$ is a quandle isomorphism from $(X,s)$ to $(Y,t)$, where $\eta:X\hookrightarrow (\Conj K)\op$ is the embedding in \eqref{eq:phi} and $\iota:K\to K$ is the inversion map $k\mapsto k\inv$. Therefore, the equivalence of the claim with Proposition \ref{prop:prelim} follows from Observation \ref{obs:iso-good} and the following direct calculations, which are straightforward (if somewhat tedious) to verify: For all elements $x\in X$ and pairs $(\rho,\psi)$ and $(\rho',\zeta)$ respectively satisfying the conditions of the claim and Proposition \ref{prop:prelim}, we have
    \[
    \theta\rho\theta\inv((x,1)\inv)=(\F(\psi)((x,1)\inv))\inv(x,1),\qquad \theta\inv\rho'\theta(x)=\F\inv(\zeta )(x)\phi(x\inv),
    \]
    so that
    \[
    \F(\psi s_x)=\F(\psi)t_{(x,1)\inv} ,\qquad \F(\psi\rho)=\F(\psi)\theta\rho\theta\inv
    \]
    and, similarly,
    \[
    \F\inv(\zeta t_{(x,1)\inv})=\F\inv(\zeta )s_x,\qquad \F\inv(\zeta \rho')=\F\inv(\zeta )\theta\inv\rho'\theta.
    \]
\end{proof}

\begin{rmk}\label{rmk:core}
    Let $\pi:X\surj \oo(X,s)$ denote the natural projection. Then Theorem \ref{thm:main2} can be reformulated as follows: Every good involution $\rho$ of $(X,s)$ is induced by a unique function ${\psi^*:\oo(X,s)\to S}$ via the formula $\rho(x)=\psi^*\pi(x)\phi\inv(x)$, which completely determines $\rho$. Moreover, the composition $\psi:=\psi^*\pi$ is constant on orbits of the action of $\rho$ on $X$.
\end{rmk}

\subsubsection{Bounds on the number of good involutions} As applications of Theorem \ref{thm:main2}, we give upper bounds on $|{\Good X}|$ and lower bounds in the case that $X=G$.

\begin{cor}[Cf.\ \cite{ta-conj}*{Cor.\ 6.12}]\label{cor:good-bound}
    The number of good involutions of $X$ satisfies the upper bound
    \[
    |{\Good X}|\leq \min(|{\Aut X}|, |S|^{|\oo(X)|}),
    \]
    where $\Aut X$ denotes the group of quandle automorphisms of $X$.
\end{cor}

\begin{proof}
    The bound $|{\Good X}|\leq |{\Aut X}|$ comes from \cite{ta-conj}*{Cor.\ 4.8}. 
    On the other hand, Remark \ref{rmk:core} shows that there are at least as many functions $\psi^*:\oo(X)\to S$ as there are good involutions $\rho$; that is, $|S|^{|\oo(X)|}\geq |{\Good X}|$.
\end{proof}

The upper bound in Corollary \ref{cor:good-bound} can be improved considerably if $X$ is connected.

\begin{cor}[Cf.\ \cite{ta-conj}*{Cor.\ 6.11}]\label{cor:main-conn}
	If $X$ is connected, then for all functions $\rho:X\to G$, $\rho$ is a good involution of $X$ if and only if there exists an element $t\in S$ such that
	\[
	\rho(x)=t\phi(x\inv)\in X
	\]
	for all $x\in X$. In particular, $|{\Good X}|\leq |S|$.
\end{cor}

\begin{proof}
	In light of Remark \ref{rmk:core}, $\psi:X\to S$ is a constant function. Therefore, taking $t:=\phi(X)$ in Theorem \ref{thm:main2} yields the claim.
\end{proof}

We now turn our attention to the case in which $X=G$.

\begin{cor}\label{cor:bnd1}
    Let $X:=G$ in \eqref{eq:s2}. Then
    \[
    |{\Good (\Conj(G,\phi))}|\geq |S|.
    \]
\end{cor}

\begin{proof}
    Since $G$ is a group, $t\phi(g\inv)\in G$ for all $t\in S$ and $g\in G$. Thus, by Theorem \ref{thm:main2}, the constant functions $\psi:G\to S$ all induce good involutions $\rho$ of $\Conj(G,\phi)$.
\end{proof}

\begin{rmk}
    If $X$ is a proper subquandle of $\Conj(G,\phi)$, then $X$ may have fewer than $|S|$ good involutions. See \cite{ta-conj}*{Subsec.\ 6.2} for examples in which $S$ is nonempty but $\Good X$ is empty.
\end{rmk}

\begin{cor}[Cf.\ \cite{taniguchi}*{Prop.\ 2.2}]\label{cor:conn-const}
    If $\Conj(G,\phi)$ is connected, then for all functions $\rho:G\to G$, $\rho$ is a good involution of $\Conj(G,\phi)$ if and only if there exists an element $t\in S$ such that \[\rho(g)=t\phi(g\inv)\] for all $g\in G$. In particular, $|{\Good (\Conj(G,\phi))}|=|S|$.
\end{cor}

\begin{proof}
    This follows directly from Corollaries \ref{cor:main-conn} and \ref{cor:bnd1}.
\end{proof}

\begin{rmk}
    For an infinite class of examples in which Corollary \ref{cor:conn-const} applies, see Corollary \ref{cor:good-taka}. Indeed, since Alexander quandles $\Alex(A,\phi)$ are both twisted conjugation quandles $\Conj(A,\phi)$ and generalized Alexander quandles $\mathrm{GAlex}(A,\phi)$, Corollary \ref{cor:conn-const} and \cite{taniguchi}*{Prop.\ 2.2} generalize Corollary \ref{cor:good-taka} in two different ways.
\end{rmk}

\section{Symmetric Alexander quandles}\label{sec:alex}
In this section, we use Theorem \ref{thm:main2} to extend several classification results of Taniguchi \cite{taniguchi} in the abelian case. 
Henceforth, let $A$ be an additive abelian group, and let $\phi\in\Aut A$ be an automorphism of $A$. Let $Q:=\Alex(A,\phi)$ denote the corresponding Alexander quandle, and let $s$ denote the quandle structure of $Q$. 

\subsection{Good involutions and fixed points} Since $Q=\Conj(A,\phi)$, taking $X:=A$ shows that the set $S$ given in \eqref{eq:s2} equals
\begin{equation}\label{eq:s22}
S=\begin{cases}
	\Fix\phi&\text{if }\phi\text{ is an involution,}\\
	\emptyset &\text{otherwise.}
\end{cases}
\end{equation}

Below, we use this observation to strengthen the statement of \cite{taniguchi}*{Thm.\ 1.3} for abelian groups.

\begin{lemma}\label{lem:q-kei}
	$Q$ is a kei if and only if $\varphi$ is an involution.
\end{lemma}

\begin{proof}
	Straightforward computation.
\end{proof}

\begin{thm}[Cf.\ \cite{taniguchi}*{Thm.\ 1.3}]\label{cor:taniguchi}
	$\Good Q$ is nonempty if and only if $Q$ is a kei. 
    
    In this case, a function $\rho:A\to A$ is a good involution of $Q$ if and only if there exists a function $\psi:A\to \Fix\phi$ such that
    \[
    \rho(a)=\psi(a)-\phi(a),\qquad \psi s_a= \psi=\psi\rho
    \]
    for all $a\in A$. In particular, $|{\Good Q}|\geq |{\Fix\phi}|$.
\end{thm}

\begin{proof}
    The claim follows immediately from combining Theorem \ref{thm:main2} with \eqref{eq:s22}, Lemma \ref{lem:q-kei}, and Corollary \ref{cor:bnd1}.
\end{proof}

\begin{example}[Cf.\ \cite{symm-quandles-2}*{Prop.\ 3.1}]
    If $Q$ is a kei, then the function $\psi(a):=a+\phi(a)$ satisfies the conditions of Theorem \ref{cor:taniguchi}, and the induced good involution $\rho$ is the identity map $\id_A$.
\end{example}

\begin{rmk}
    There exist infinitely many Alexander quandles $Q$ such that $|{\Good Q}|= |{\Fix\phi}|$ and infinitely many $Q$ such that $|{\Good Q}|> |{\Fix\phi}|$; see Corollary \ref{cor:good-taka} and Proposition \ref{prop:dihedral}.
\end{rmk}

\subsection{Connected Alexander quandles}
In this subsection, we discuss good involutions of connected Alexander quandles $Q$. 
First, define the function $\alpha:A\to A$ by
\[
\alpha(a):=a-\phi(a)
\]
for all $a\in A$. Recall from Example \ref{lem:conn-alex} that $Q$ is connected if and only if $\alpha$ is surjective. Also, recall that $A$ is called \emph{$2$-divisible} if the function $a\mapsto 2a$ is a surjection from $A$ to $A$.

\begin{lemma}\label{lem:conn-alex-2}
    If $\phi$ is an involution and $\alpha$ is surjective, then $\phi$ is the inversion map $a\mapsto -a$.
\end{lemma}

\begin{proof}
    Let $a\in A$. Since $\alpha$ is surjective, there exists an element $b\in A$ such that $\alpha(b)=a$, so that
    \[
    \phi(a)=\phi\alpha(b)=\phi(b)-\phi^2(b)=\phi(b)-b=-\alpha(b)=-a.
    \]
    Since $a\in A$ was arbitrary, the claim follows.
\end{proof}

\begin{rmk}
It is well-known that if $A$ is finite, then $Q$ is connected if and only if $\phi$ is fixed-point-free (that is, $\Fix \phi=\{0\}$) \cite{sim}.
    However, neither implication holds in general if $A$ is infinite.

    To see one direction, let $A$ be the direct sum $\bigoplus_{n\geq 1}\Z$, and define $\phi:A\to A$ by
    \[
    (k_1,k_2,k_3,\dots) \mapsto (k_1+k_2,k_2+k_3,k_3+k_4,\dots).
    \]
    Then $\phi$ is an automorphism of $A$, and $\alpha:A\to A$ is surjective, but $\phi$ is not fixed-point-free:
    \[
    \Fix\phi=\{(k,0,0,\dots)\mid k\in\Z\}.
    \]

    Conversely, let $A$ be any $2$-torsionless abelian group that is not $2$-divisible, and let $\varphi$ be the inversion map $a\mapsto -a$. Then $\varphi$ is a fixed-point-free automorphism of $A$, but $\alpha:A\to A$ is the nonsurjective mapping $a\mapsto 2a$. Indeed, the converse of Lemma \ref{lem:conn-alex-2} holds if and only if $A$ is $2$-divisible.
\end{rmk}

\begin{lemma}\label{lem:taka-conn}
    The Takasaki kei $T(A)$ is connected if and only if $A$ is $2$-divisible.
\end{lemma}

\begin{proof}
    This is a special case of Example \ref{lem:conn-alex}.
\end{proof}

\subsubsection{Connected Alexander kei are Takasaki kei}
The following completely characterizes involutory, connected Alexander quandles.

\begin{thm}\label{thm:conn-q}
    For all Alexander quandles $Q=\Alex(A,\phi)$, the following are equivalent:
    \begin{enumerate}
        \item $Q$ is a connected kei.
        \item $A$ is $2$-divisible, and $Q$ equals the Takasaki kei $T(A)$.
    \end{enumerate}
\end{thm}

\begin{proof}
Note that if $\phi$ is the inversion map $a\mapsto -a$, then $\alpha$ is the mapping $a\mapsto 2a$.
    Hence, the claim follows directly from Example \ref{lem:conn-alex} and Lemmas \ref{lem:q-kei}, \ref{lem:conn-alex-2}, and \ref{lem:taka-conn}.
\end{proof}

\begin{rmk}
    Theorem \ref{thm:conn-q} does not hold for generalized Alexander quandles $\mathrm{GAlex}(G,\phi)$. 
    For example, let $G:=A_5$ be the alternating group on five letters, let $a:=(12)(34)$, and consider the inner automorphism $\phi:A_5\to A_5$ defined by $g\mapsto aga$. Then the generalized Alexander quandle $\operatorname{GAlex}(A_5,\phi)$ is a kei because $\phi$ is an involution, and a quick \texttt{GAP} \cite{GAP4} verification shows that $\operatorname{GAlex}(A_5,\phi)$ is connected. Of course, $\operatorname{GAlex}(A_5,\phi)$ is not a Takasaki kei.
\end{rmk}

Using Theorem \ref{thm:conn-q}, we strengthen the statement of \cite{taniguchi}*{Prop.\ 2.2} for abelian groups.

\begin{cor}[Cf.\ \cite{taniguchi}*{Prop.\ 2.2}]\label{cor:good-taka}
	Suppose that $Q=\Alex(A,\phi)$ is involutory and connected, and let $T$ be the subgroup of $A$ whose elements are involutions. 
    
	Then for all functions $\rho:A\to A$, $\rho$ is a good involution of $Q$ if and only if there exists an element $t\in T$ such that $\rho(a)=t+a$ for all $a\in A$. In particular, $|{\Good Q}|=|T|$.
\end{cor}

\begin{proof}
    By Theorem \ref{thm:conn-q}, we have $\phi(a)=-a$ for all $a\in A$, so $T=\Fix\phi$ equals the set $S$ in \eqref{eq:s22}. Hence, the claim is a special case of Corollary \ref{cor:conn-const}.
\end{proof}

\begin{rmk}
    Under the hypotheses of Corollary \ref{cor:good-taka}, Theorem \ref{thm:conn-q} states that $A$ must be $2$-divisible. 
    If $A$ is \emph{uniquely} $2$-divisible (that is, if $\alpha$ is a bijection), then $T=\{0\}$ in Corollary \ref{cor:good-taka}, so $\Good Q=\{\id_A\}$. 
    
    However, $A$ may be $2$-divisible but not uniquely $2$-divisible. For example, if $A$ equals $\mathbb{Q}/\Z$ or the Pr{\"u}fer $2$-group, then $A$ satisfies the conditions of Theorem \ref{thm:conn-q}. In these examples, $T\cong\Z/2\Z$ in Corollary \ref{cor:good-taka}, so $|{\Good Q}|=2$.
\end{rmk}

\section{Symmetric linear quandles}\label{sec:linear}
In this section, we apply the results of Section \ref{sec:alex} to study good involutions of \emph{linear quandles}, which are Alexander quandles of finite cyclic groups. 

\subsection{Preliminaries}
We give a brief overview of linear quandles. See \cite{alex} for classification results for linear quandles and \cite{fox} for applications of linear quandles to knot theory.

\begin{definition}[\cite{alex}]
    Let $n\in\Z^+$ be a positive integer, and let $k\in(\Z/n\Z)^\times$. Define the \emph{linear quandle} $\Lambda_{n,k}$ to be the Alexander quandle $\Alex(\Z/n\Z,\phi_k)$, where $\phi_k(m):=km$ for all $m\in\Z/n\Z$.
\end{definition}

\begin{rmk}[\citelist{\cite{alex}\cite{hou}}]
    In light of Remark \ref{rmk:module}, $\Lambda_{n,k}$ can be regarded as the $\Z[t^{\pm 1}]$-module
    \[
    \Z[t^{\pm 1}]/(n,t-k).
    \]
\end{rmk}

\begin{example}
    For all $n\in\Z^+$, the linear quandle $\Lambda_{n,-1}$ equals the dihedral quandle $R_n$.
\end{example}

\begin{example}\label{ex:trivial-lin}
    A linear quandle $\Lambda_{n,k}$ is a trivial quandle if and only if $k\equiv 1\pmod{n}$.
\end{example}

\begin{obs}\label{obs:gcd}
    Given a linear quandle $\Lambda_{n,k}$, let $d:=\gcd(n,k-1)$. Then the fixed points of the action of $\phi_k$ on $\Z/n\Z$ are
    \[
    \Fix\phi_k=\langle n/d\rangle\cong\Z/d\Z.
    \]
\end{obs}

\subsection{Good involutions of linear quandles}
The remainder of this paper is dedicated to studying good involutions of linear quandles. 
Since good involutions of trivial quandles are already classified (see \cite{symm-quandles-2}*{Prop.\ 3.1}), we will only turn our attention to linear quandles $\Lambda_{n,k}$ with $k\not\equiv 1\pmod{n}$.

\begin{prop}\label{prop:n-nontriv}
    Let $n\in\Z^+$ be a positive integer. Then the number of nontrivial linear quandles of order $n$ that have good involutions equals the number of elements of order $2$ in $(\Z/n\Z)^\times$. (This number is $1$ less than the $n$th term of OEIS sequence A$060594$ \cite{oeis2}.)
\end{prop}

\begin{proof}
    Thanks to the group isomorphism $\Aut \Z/n\Z\cong(\Z/n\Z)^\times$, the claim follows from Lemma \ref{lem:q-kei}, Theorem \ref{cor:taniguchi}, and Example \ref{ex:trivial-lin}.
\end{proof}

\begin{rmk}
    The linear quandles counted in Proposition \ref{prop:n-nontriv} may not be pairwise nonisomorphic; for example, see Remark \ref{rmk:isom}.
\end{rmk}

\begin{example}
    It is well-known that $-1$ is the only element of order $2$ in $(\Z/n\Z)^\times$ if and only if $n$ equals $4$, a power of an odd prime, or twice the power of an odd prime. In this case, Proposition \ref{prop:n-nontriv} states that the only nontrivial linear quandle of order $n$ that has good involutions is the dihedral quandle $R_n=\Lambda_{n,-1}$; cf.\ Proposition \ref{prop:dihedral} and Tables \ref{tab:linear} and \ref{tab2}.
\end{example}

\subsubsection{Computer search}
Using Observation \ref{obs:gcd} and Theorem \ref{cor:taniguchi}, we implemented a \texttt{GAP} \cite{GAP4} program that exhaustively searches for good involutions of all linear quandles $\Lambda_{n,k}$ of a given order $n\geq 3$ counted in Proposition \ref{prop:n-nontriv}. We provide our code and raw data for all $3\leq n\leq 29$ in a GitHub repository \cite{code}. 

Table \ref{tab:linear} enumerates good involutions of all nontrivial linear quandles $\Lambda_{n,k}$ up to order $29$. The entries of Table \ref{tab:linear} were obtained using \texttt{GAP} program or, in the cases $(n,k)=(24,13)$ and $(n,k)=(28,15)$, using Theorem \ref{thm:diverge}. Outside of these two cases, we provide explicit mappings for the good involutions counted in Table \ref{tab:linear} in the same GitHub repository containing our \texttt{GAP} code \cite{code}. For descriptions of several of the subsequences recorded in Table \ref{tab:linear}, see Proposition \ref{prop:dihedral}, Theorem \ref{thm:diverge}, and Conjecture \ref{conjecture}.

Summing the entries of Table \ref{tab:linear}, Table \ref{tab2} displays the number of good involutions of all linear quandles of a given order $n\leq 29$. Note that the counts in Tables \ref{tab:linear} and \ref{tab2} are up to equality rather than up to isomorphism; cf.\ Remark \ref{rmk:isom}.

\begin{longtable}{cccc}
\caption{Number of good involutions of nontrivial linear quandles $\Lambda_{n,k}$ up to order $29$.}
\label{tab:linear} \\

\multicolumn{1}{c}{$(n,k)$} & \multicolumn{1}{c}{$|{\Good(\Lambda_{n,k})}|$} \\ \hline 
\endfirsthead

\caption[]{(continued)}\\
\multicolumn{2}{c}%
{}\\
\multicolumn{1}{c}{$(n,k)$} & \multicolumn{1}{c}{$|{\Good(\Lambda_{n,k})}|$} \\ \hline 
\endhead
\endfoot\endlastfoot

$(3,-1)$                                     & 1                                        \\
$(4,-1)$                                      & 4                                         \\
$(5,-1)$                                      & 1                                         \\
$(6,-1)$                                      & 2                                         \\
$(7,-1)$                                      & 1                                         \\
$(8,-1)$                                      & 4                                         \\
$(8,3)$                                      & 4                                         \\
$(8,5)$                                      & 36                                        \\
$(9,-1)$                                     & 1                                      \\
$(10,-1)$                                    & 2                                       \\
$(11,-1)$                                    & 1                                         \\
$(12,-1)$                                      & 4                                       \\
$(12,5)$                                   & 10                                        \\
$(12,7)$                                   & 400                                      \\
$(13,-1)$                                   & 1                                        \\
$(14,-1)$                                   & 2                                         \\
$(15,-1)$                             & 1                                       \\
$(15,4)$                                     & 4                                      \\
$(15,11)$                                   & 26                                        \\
$(16,-1)$                                    & 4                                        \\
$(16,7)$                                    & 4                                       \\
$(16,9)$                                     & 5776                                       \\
$(17,-1)$                                    & 1                                        \\
$(18,-1)$                                     & 2                                         \\
$(19,-1)$                                    & 1                                        \\
$(20,-1)$                                    & 4                                       \\
$(20,9)$                                    & 10                                         \\
$(20,11)$     & 97344                                \\
$(21,-1)$                  & 1                                \\
$(21,8)$                   & 232                                \\
$(21,13)$                  & 4                              \\
$(22,-1)$                                   & 2                                      \\
$(23,-1)$                             & 1                                    \\
$(24,-1)$                              & 4                                     \\
$(24,5)$     & 36                               \\
$(24,7)$                     & 400\\
$(24,11)$                       & 4                              \\
$(24,13)$     & 1915456                                \\
$(24,17)$                                     & 764                                      \\
$(24,19)$                                   & 400                                        \\
$(25,-1)$                                    & 1                                       \\
$(26,-1)$                     & 2                                  \\
$(27,-1)$     & 1                                        \\
$(28,-1)$                      & 4                                   \\
$(28,13)$                                    & 10                                       \\
$(28,15)$                               & 42406144                                    \\
$(29,-1)$                                     & 1                                         
\end{longtable}

\begin{table}[h]
\caption{Number of good involutions of all linear quandles $\Lambda_{n,k}$ of a given order $n$.}
\label{tab2}
\begin{tabular}{cc}
$n$ & Symmetric quandles $(\Lambda_{n,k},\rho)$ \\ \hline
3     & 1                          \\
4     & 4                          \\
5     & 1                          \\
6     & 2                          \\
7     & 1                          \\
8     & 44                         \\
9     & 1                          \\
10    & 2                          \\
11    & 1                          \\
12    & 414                        \\
13    & 1                          \\
14    & 2                          \\
15    & 31                         \\
16    & 5784                       \\
17    & 1                          \\
18    & 2                          \\
19    & 1                          \\
20    & 97358                      \\
21    & 237                        \\
22    & 2                          \\
23    & 1                          \\
24    & 1917064                    \\
25    & 1                          \\
26    & 2                          \\
27    & 1                          \\
28    & 42406158                   \\
29    & 1                         
\end{tabular}
\end{table}

\subsubsection{General results}

In the following, we consider several particular families of linear quandles counted in Proposition \ref{prop:n-nontriv}. First, we give an alternative proof of a well-known result of Kamada and Oshiro \cite{symm-quandles-2}.
\begin{prop}[\cite{symm-quandles-2}*{Thm.\ 3.2}]\label{prop:dihedral}
    Let $n\in\Z^+$ be a positive integer. Then the number of good involutions of the dihedral quandle $R_n$ is
    \[
    |{\Good R_n}|=\begin{cases}
        1&\text{if }n\text{ is odd,}\\
        2&\text{if }n=2m\text{ with }m\text{ odd,}\\
        4&\text{otherwise.}
    \end{cases}
    \]
\end{prop}

\begin{proof}
    Let $A:=\Z/n\Z$, and let $\phi_{-1}:A\to A$ be the inversion map $x\mapsto-x$. If $n$ is odd, then $\Fix\phi_{-1}=\{0\}$, so the claim follows directly from Theorem \ref{cor:taniguchi}.

    Otherwise, write $n=2m$ with $m\in\Z^+$. Evidently, the connected components of $R_n$ are $\oo(R_n)=\{2A,1+2A\}$. On the other hand, $\Fix \phi_{-1}=\{0,m\}$.
    
    By Corollary \ref{cor:good-bound}, $R_n$ has at most four good involutions $\rho_i$. By Theorem \ref{cor:taniguchi} (cf.\ Remark \ref{rmk:core}), each $\rho_i$ is induced by exactly one of the following functions $\psi_i:A\to \Fix\phi_{-1}$ via the formula $\rho_i(a):=\psi(a)+a$:
    \[
    \psi_1(a):= 0,\quad \psi_2(a):=m,\quad \psi_3(a):=\begin{cases}
        0&\text{if }a\text{ is even,}\\
        m &\text{otherwise,}
    \end{cases}\quad
    \psi_4(a):=\begin{cases}
        m&\text{if }a\text{ is even,}\\
        0 &\text{otherwise.}
    \end{cases}
    \]
    For $i=1,2$, it is clear that $\psi s_a=\psi_i=\psi_i\rho_i$ for all $a\in A$. A straightforward check shows that the same holds for $i=3$ if and only if $m$ is even, and similarly for $i=4$. Hence, Theorem \ref{cor:taniguchi} completes the proof.
\end{proof}

In contrast to Proposition \ref{prop:dihedral}, the next theorem shows that the sequence $(|{\Good \Lambda_{4n,2n+1}}|)_{n\geq 1}$ diverges quite rapidly as $n$ tends toward infinity. Our \texttt{GAP} program verified this theorem for all $n\leq 5$.\footnote{On the author's computer, we were unable to complete the computation for $n=6,7$ using our \texttt{GAP} program. This is why the explicit elements of and $\Good \Lambda_{24,13}$ and $\Good \Lambda_{28,15}$ do not appear in our GitHub repository \cite{code} even though their cardinalities appear in Table \ref{tab:linear}.} We begin with a well-known combinatorial lemma.

\begin{lemma}\label{lemma:trans}
    Let $n\in\Z^+$ be a positive integer, let $X$ be a set of size $n$, and let $0\leq i\leq \lfloor n/2\rfloor$ be an integer. Then the number of involutions in $S_X$ having $i$ transpositions equals
    \[
    \binom{n}{2i}(2i-1)!!=\frac{n!}{(n-2i)!i!2^i}.
    \]
\end{lemma}

\begin{thm}\label{thm:diverge}
    Let $n\in\Z^+$ be a positive integer, and let $k:=2n+1$. Then the number of good involutions of the linear quandle $\Lambda_{4n,k}$ is
    \[
    |{\Good \Lambda_{4n,k}}|=\left(\sum^{\lfloor n/2\rfloor}_{i=0}\frac{n!}{(n-2i)!i!}2^{n-2i}\right)^2.
    \]
    In other words, the integer sequence $(|{\Good\Lambda_{4n,k}}|)_{n\geq 1}$ is precisely the OEIS sequence A$202828$ \cite{oeis}.
\end{thm}

\begin{proof}
    Let $A:=\Z/4n\Z$, and recall from Theorem \ref{cor:taniguchi} that every good involution $\rho\in\Good \Lambda_{n,k}$ is induced by a unique function $\psi^*:\oo(\Lambda_{4n,k})\to\Fix\phi_k$.
    By \cite{orbits}*{Thm.\ 1}, $\Lambda_{4n,k}$ has exactly $2n$ connected components, namely the orbits
    \begin{equation}\label{eq:oom}
        \oo_m:=\{m,m+2n\},\qquad 0\leq m<2n.
    \end{equation}
    On the other hand, $\Fix \phi_k=2A$.

    Let $\pi:A\surj\oo(\Lambda_{4n,k})$ be the natural projection. By Theorem \ref{cor:taniguchi} (cf.\ Remark \ref{rmk:core}), it suffices to count all functions $\psi^*:\oo(\Lambda_{4n,k})\to\Fix\phi_k$ such that $\psi=\psi\rho$, where $\psi:A\to \Fix\phi_k$ and $\rho:A\to A$ are defined by
    \begin{equation}\label{eq:rho}
        \psi:=\psi^*\pi, \qquad\rho(m):=\psi(m)-km
    \end{equation}
    for all $m\in A$. In particular, $\psi$ must be constant on each orbit \[\mathcal{R}_m:=\{m,\rho(m)\}\] 
    of the $\Z/2\Z$-action of $\rho$ on $A$. Moreover, \eqref{eq:rho} implies that $\psi(m)\equiv\rho(m)+m\pmod{2n}$ or, equivalently,
    \[
    \psi(m)\in\{\rho(m)+m,\rho(m)+m+2n\}.
    \]
    Conversely, both of these values of $\psi(m)$ achieve the equality $\psi(m)=\psi\rho(m)$, with $\rho$ as defined in \eqref{eq:rho}. Hence, for each connected component $\oo_m\in\oo(\Lambda_{4n,k})$ and for each orbit $\mathcal{R}_m$, there exist exactly two assignments $\oo_m\mapsto \psi^*(\oo_m)$ that achieve the desired equality $\psi(m)=\psi\rho(m)$.

    Given $\psi^*$ satisfying the above conditions,  \eqref{eq:rho} shows that $\rho$ preserves parity; that is, $\rho$ restricts to an involution of $2A$ and similarly for the coset $1+2A$. Let $i$ be the number of orbits $\mathcal{R}_m$ such that $m$ is even and $|\mathcal{R}_m|=2$, and let $\psi^*_{2A}$ denote the restriction of $\psi^*$ to connected components of the form $\oo_m$ with $m$ even. Then $2A$ has exactly $n-i$ orbits $\mathcal{R}_m$ under the action of $\rho$.
    
    By the above discussion, each orbit $\mathcal{R}_m$ contributes exactly two values for $\psi^*(m)$ that achieve the desired equalities. This yields exactly $2^{n-i}$ functions $\psi^*_{2A}$ that induce an involution $\rho|_{2A}$ of $2A$ satisfying the above constraints and having a given orbit decomposition with $i$ orbits of size $2$. 
    
    The number of all possible orbit decompositions of $2A$ under the action of $\rho|_{2A}$ with $i$ orbits of size $2$ is given by Lemma \ref{lemma:trans}. Multiplying by $2^{n-i}$ as above and summing over all possible values of $i$, we deduce that there are exactly
    \[
    \sum^{\lfloor n/2\rfloor}_{i=0}\frac{n!}{(n-2i)!i!2^i}2^{n-i}=\sum^{\lfloor n/2\rfloor}_{i=0}\frac{n!}{(n-2i)!i!}2^{n-2i}
    \]
    functions $\psi^*_{2A}$ that induce an involution $\rho|_{2A}$ of $2A$ satisfying the above constraints. The same goes for functions $\psi^*_{1+2A}$ that induce an involution $\rho|_{1+2A}$ of $1+2A$ satisfying the above constraints. Since $\rho$ must preserve parity, any two choices of $\psi^*_{2A}$ and $\psi^*_{1+2A}$ both satisfying the above conditions yield a function $\psi^*$ satisfying the conditions of Theorem \ref{cor:taniguchi}, so the claim follows.
\end{proof}

\begin{rmk}
   Theorem \ref{thm:diverge} confirms that the limit superior of the integer sequence in the second column of Table \ref{tab2} equals infinity.
\end{rmk}

We conclude with the following conjecture, which our \texttt{GAP} program confirmed for $n=3,5,7$.

\begin{conj}\label{conjecture}
    Let $n\geq 3$ be an odd integer, and let $k:=2n-1$. Then $|{\Good\Lambda_{4n,k}}|=10$.
\end{conj}

We note that in the setting of Conjecture \ref{conjecture}, we have $\Fix\phi_k=\{0,n,2n,3n\}$, and $\oo(\Lambda_{4n,k})$ is given by \eqref{eq:oom}.

\begin{rmk}\label{rmk:isom}
    If we change the hypotheses of Conjecture \ref{conjecture} to make $n$ even, then $\Lambda_{4n,k}$ is isomorphic to the dihedral quandle $R_{4n}$ by \cite{alex}*{Cor.\ 2.2}. Hence, Proposition \ref{prop:dihedral} implies that $|{\Good\Lambda_{4n,k}}|=4$.
\end{rmk}

\bibliographystyle{amsplain}

\end{document}